\documentclass[11pt]{article}
\usepackage[a4paper,margin=1.1in]{geometry}
\usepackage{amsmath,amssymb,amsthm,mathtools}
\usepackage{enumitem}
\usepackage{hyperref}
\usepackage{microtype}
\usepackage{mathrsfs}
\usepackage{xcolor}
\usepackage{geometry}
\hypersetup{hidelinks}

\numberwithin{equation}{section}

\newtheorem{theorem}{Theorem}[section]
\newtheorem{proposition}[theorem]{Proposition}
\newtheorem*{proposition*}{Proposition}
\newtheorem{lemma}[theorem]{Lemma}
\newtheorem{assumption}[theorem]{Assumption}
\newtheorem{remark}[theorem]{Remark}
\newtheorem{definition}[theorem]{Definition}

\newcommand{\R}{\mathbb R}
\newcommand{\N}{\mathbb N}
\newcommand{\E}{\mathbb E}
\newcommand{\PP}{\mathbb P}

\newcommand{\1}{\mathbf 1}

\newcommand{\norm}[1]{\left\lVert #1 \right\rVert}
\newcommand{\ip}[2]{\left\langle #1,#2\right\rangle}
\newcommand{\OO}{\mathcal O_0}
\newcommand{\OOt}{\mathcal O_t}
\newcommand{\Ot}{\mathcal O_t}
\newcommand{\pa}{\partial}
\newcommand{\divy}{\operatorname{div}_y}
\newcommand{\varephi}{\varphi}

\title{Well-posedness and regularity of stochastic heat equations on moving domains}
\author{
Chongyang Ren\textsuperscript{1},
Tusheng Zhang\textsuperscript{1,2}
}
\date{}

\begin{document}
\maketitle
\footnotetext[1]{%
School of Mathematics, University of Science and Technology of China,
Hefei, China.
Email: \texttt{rcy.math@ustc.edu.cn} (Chongyang Ren).
}

\footnotetext[2]{%
Department of Mathematics, University of Manchester,
Manchester M13 9PL, United Kingdom.
Email: \texttt{tusheng.zhang@manchester.ac.uk}.
}

\begin{abstract}
In this paper we investigate stochastic heat equations driven by multiplicative noise on moving domains. We establish the well-posedness within a nonhomogeneous variational framework. Furthermore, by combining stochastic De Giorgi iteration with Dirichlet parabolic estimates, we obtain the H\"older regularity of the solutions.
\end{abstract}

\section{Introduction}
In this paper, we are concerned with the well-posedness and regularity of stochastic heat equations driven
by multiplicative noise on a family of time-dependent bounded domains
$
\{\mathcal O_t\}_{0\leq t\leq T}\subset\R^d
$
 which are given as follows,
\begin{equation}
\left\{
\begin{aligned}
&\frac{\partial u}{\partial t}(t,x)
=
\Delta_x u(t,x)+f(t,x,u(t,x))
+ g_i(t,x,u(t,x))\dot w_t^i,
&&x\in\OOt,\ 0<t\leq T,\\
&u(t,x)=0,
&&x\in\partial\OOt,\ 0<t\leq T,\\
&u(0,x)=u_0(x),
&&x\in\OO .
\end{aligned}
\right.
\label{eq:moving-spde}
\end{equation}
Here $\Delta_x u(t,x)$ means the action of the Dirichlet Laplacian on the domain $\mathcal O_t$. The processes $\{w^i\}_{i\in\mathbb N}$ are independent standard Brownian motions on a filtered probability space $(\Omega,\mathcal F,\{\mathcal F_t\}_{t\ge0},\mathbb P)$. The coefficients $f$, $g=(g_i)_{i\in\mathbb N}$ are random. Throughout the paper, we adopt the Einstein summation convention:
repeated indices are summed over the indices.

The study of partial differential equations on moving domains is part of the broader problem of understanding equations on non-static spaces. Such problems arise naturally in fluid mechanics, free-boundary models, biological growth, diffusion in moving media, reaction-diffusion systems on evolving spatial regions, and materials science. Passing from stationary domains to domains that evolve with time often leads to more realistic models; typical examples include surface dissolution of binary alloys, pattern formation and chemotaxis on evolving biological surfaces, cell motility, and elastic membranes. We refer to \cite{BarreiraElliottMadzvamuse2011,EilksElliott2008,ElliottStinnerVenkataraman2012,GarckeLamStinner2014} for some representative applications. 

Random forcing makes the moving-domain problem substantially more delicate. Even when the deterministic motion of the domain is known, the state space changes with time, while the stochastic integral must still be interpreted in a fixed probabilistic framework. For stochastic partial differential equations (SPDEs) on time-dependent domains, \cite{WangZhaiZhang2022} studied stochastic two-dimensional Navier--Stokes equations on time-dependent domains with additive noise, and \cite{PanWangZhaiZhang2025} established well-posedness for stochastic heat equations on one-dimensional moving intervals driven by multiplicative noise. We also note that the authors in \cite{PanWangZhaiZhang2026} formulate a general setting, the so-called nonhomogeneous monotonicity, under which the well-posedness of SPDEs on moving domains can be established.

Regularity of solutions to stochastic parabolic equations is a central issue in the theory of SPDEs. On fixed spatial domains, the linear theory has been developed from several points of view: the $W^{k,p}$-theory has been well developed in \cite{Krylov1996}, and semigroup methods were studied in \cite{DaPratoZabczyk1992,BrzezniakVanNeervenVeraarWeis2008}. Results concerning nonlinear stochastic parabolic equations have also been established in \cite{DebusscheDeMoorHofmanova2015}. In the direction of H\"older regularity, stochastic De Giorgi iteration was introduced in \cite{HsuWangWang2017} to obtain pathwise H\"older continuity under very weak regularity assumptions on the coefficients. 

The purpose of the present paper is to establish well-posedness and H\"older regularity for stochastic heat equations with multiplicative noise on d-dimensional moving domains. Compared with \cite{PanWangZhaiZhang2025}, the present work has two main new features. First, the approach in \cite{PanWangZhaiZhang2025} relies essentially on the evolving eigenbasis of the one-dimensional Dirichlet Laplacian. Consequently, that method is intrinsically restricted to moving intervals and does not extend directly to moving domains in higher dimensions. In contrast, we treat bounded moving domains in $\R^d$ by pulling the equation back to a fixed reference domain and applying the nonhomogeneous variational framework. Second, beyond the well-posedness theory, we prove H\"older regularity for the solution of stochastic heat equations on moving domains using stochastic De Giorgi iteration together with Dirichlet parabolic estimates.

The rest of the paper is organized as follows. Section~\ref{sec:main-result} states the main results  and explains the proof strategy. Section~\ref{sec:variational-framework} proves well-posedness using the nonhomogeneous variational framework developed in \cite{PanWangZhaiZhang2026}. Section~\ref{sec:fixed-regularity} proves the H\"older estimate for solutions of SPDEs on a fixed domain. Section~\ref{sec:moving-regularity} transfers this estimate back to the moving domains and completes the proof of the regularity of the solution.

\section{Main result and proof strategy}
\label{sec:main-result}
In this section, we state the main result and briefly describe the proof strategy.
The moving domains are generated by a deterministic flow. More precisely, there is a map
$$
r:[0,T]\times\OO\to\R^d
$$
such that $r(0,y)=y$ and $r(t,\cdot):\OO\to\OOt$ is a diffeomorphism for each $t$. We use $y\in\OO$ for the reference variable and $x=r(t,y)\in\OOt$ for the space variable. The inverse map is denoted by
\begin{equation*}
\rho(t,x)=r(t,\cdot)^{-1}(x),\qquad x\in\OOt .
\end{equation*}
Thus
\begin{equation*}
\rho(t,r(t,y))=y,\qquad r(t,\rho(t,x))=x .
\end{equation*}
The Jacobian of the flow is

$$
J(t,y)=\left|\det D_y r(t,y)\right|.
$$

The basic assumptions are as follows:
\begin{assumption}
\label{ass:moving-domain}
The maps $r$ and $\rho$ satisfy
\begin{equation*}
\sup_{0\leq t\leq T}
\Big(
\norm{r(t,\cdot)}_{C^3(\overline\OO)}
+\norm{\pa_t r(t,\cdot)}_{C^2(\overline\OO)}
\Big)<\infty .
\end{equation*}
\end{assumption}
\begin{assumption}
\label{ass:coefficients}
There exist constants $L,\Lambda, K_{0}>0$ such that, for every $t\in[0,T]$, $x\in\OOt$, and $z,z'\in\R$, we have
\begin{align}
|f(t,x,&z)-f(t,x,z')|
+\norm{g(t,x,z)-g(t,x,z')}_{\ell^2}
\leq L|z-z'|,
\label{eq:fg-lipschitz}\\
&|f(t,x,z)|+\norm{g(t,x,z)}_{\ell^2}
\leq K_{0}+\Lambda |z|.
\label{eq:moving-growth}
\end{align}
\end{assumption}

\begin{remark}
\begin{itemize}
\item[(i)] Assumption 2.1 also implies that the inverse mapping $\rho$ also satisfies
\begin{equation}\label{00}\sup_{0\leq t\leq T}
\Big(\norm{\rho(t,\cdot)}_{C^3(\overline\OOt)}
+\norm{\pa_t\rho(t,\cdot)}_{C^2(\overline\OOt)}
\Big)<\infty .
\end{equation}

\item[(ii)]
For each fixed $y\in\OO$, the function
$t\mapsto\det D_y r(t,y)$ is continuous and never vanishes, since
$r(t,\cdot)$ is a diffeomorphism. Moreover,
$r(0,y)=y$ implies $\det D_y r(0,y)=1$. Hence, 
\[
\det D_y r(t,y)>0
\qquad
\text{for all }(t,y)\in[0,T]\times\OO,
\]
from which we have
\[
J(t,y)=\left|\det D_y r(t,y)\right|=\det D_y r(t,y).
\]
\end{itemize}
\end{remark}

Apart from Assumptions~\ref{ass:moving-domain} and \ref{ass:coefficients}, we also assume that the reference domain $\OO$ is $C^{1,1}$ and that the initial datum $u_0\in L^2(\OO)$ is deterministic. To state the main result of the paper, we introduce some noations. Let
\begin{equation*}
Q_{T}=\{(t,x):0\le t\le T,\ x\in\OOt\}.
\end{equation*}
Define
\begin{equation}
H_0^1(\OO)
:=
\overline{C_c^\infty(\OO)}^{\,H^1(\OO)}.
\label{eq:H01-definition}
\end{equation}
For a deterministic function $\varephi$ on $Q_{T}$, set
\begin{equation*}
\widehat\varephi(t,y):=\varephi(t,r(t,y)),
\qquad (t,y)\in[0,T]\times\overline\OO,
\end{equation*}
and define the Dirichlet test class
\begin{equation*}
C_D^2(Q_{T})
:=
\left\{
\varephi: \varephi\in C^2([0,T]\times\OO),\quad
\varephi(t,\cdot)|_{\partial\Ot}=0
\text{ for every }t\in[0,T]
\right\}.
\end{equation*}
Under Assumption~\ref{ass:moving-domain}, this definition is
equivalent to the pullback formulation
\[
\widehat\varphi(t,y):=\varphi(t,r(t,y))
\in C^2([0,T]\times \OO),
\qquad
\widehat\varphi(t,\cdot)|_{\partial\OO}=0.
\]
Indeed, the converse relation is
\(
\varphi(t,x)=\widehat\varphi(t,\rho(t,x)).
\)
For a function $u$ defined on $Q_{T}$, let
\[
\widehat{u}(t,y):=u(t,r(t,y)),
\qquad (t,y)\in[0,T]\times\OO.
\]
We define
\begin{align}
C([0,T];L^2(\mathcal O_{\cdot}))
:={}&
\Bigl\{
u=\{u(t)\}_{t\in[0,T]}:
u(t)\in L^2(\mathcal O_t)\ \text{for every }t\in[0,T],
\nonumber\\
&\hspace{35mm}
\widehat{u}(\cdot)\in C([0,T];L^2(\OO))
\Bigr\},
\label{eq:moving-C-space}
\\
L^2(0,T;H_0^1(\mathcal O_{\cdot}))
:={}&
\Bigl\{
u=\{u(t)\}_{t\in[0,T]}:
u(t)\in H_0^1(\mathcal O_t)\ \text{for a.e. }t\in(0,T),
\nonumber\\
&\hspace{35mm}
\widehat{u}(\cdot)\in L^2(0,T;H_0^1(\OO))
\Bigr\}.
\label{eq:moving-L2-H01-space}
\end{align}
Here is the definition of the solution to equation
\eqref{eq:moving-spde}.
\begin{definition}
\label{def:moving-solution}
An adapted random field $u$ on $Q_{T}$ is called a weak solution of equation \eqref{eq:moving-spde} if
\begin{equation}
u\in L^2(\Omega;C([0,T];L^2(\mathcal O_{\cdot})))
\cap
L^2(\Omega;L^2(0,T;H_0^1(\mathcal O_{\cdot}))),
\label{eq:moving-solution-space}
\end{equation}
and
for every deterministic $\varephi\in C_D^2(Q_{T})$ and every
$t\in[0,T]$, the following identity holds almost surely:
\begin{align}
\int_{\OOt}u(t,x)\varephi(t,x)\,dx
&=
\int_{\OO}u_0(x)\varephi(0,x)\,dx
-\int_0^t\int_{\mathcal O_s}\nabla_xu(s,x)\cdot\nabla_x\varephi(s,x)\,dx\,ds
\nonumber\\
&\quad
+\int_0^t\int_{\mathcal O_s}u(s,x)\partial_s\varephi(s,x)\,dx\,ds
+\int_0^t\int_{\mathcal O_s}f(s,x,u(s,x))\varephi(s,x)\,dx\,ds
\nonumber\\
&\quad
+\int_0^t\int_{\mathcal O_s}g_i(s,x,u(s,x))\varephi(s,x)\,dx\,dw_s^i .
\label{eq:moving-weak-form}
\end{align}
\end{definition}

 For $T_0>0$ with $2T_0\leq T$, define the moving cylinder
$$
Q_{T_0,2T_0}
=
\{(t,x):T_0\leq t\leq 2T_0,\ x\in\overline{\OOt}\}.
$$
For a function $v$ on $Q_{T_0,2T_0}$, set
\begin{align*}
\|v\|_{C^{\alpha/2,\alpha}(Q_{T_0,2T_0})}
&:=
\sup_{(t,x)\in Q_{T_0,2T_0}}|v(t,x)|
+
\sup_{\substack{(t,x),(s,z)\in Q_{T_0,2T_0}\\ (t,x)\neq(s,z)}}
\frac{|v(t,x)-v(s,z)|}
{|t-s|^{\alpha/2}+|x-z|^\alpha}.
\end{align*}

\begin{theorem}\label{thm:physical-holder-transfer}
Under Assumptions 2.1 and 2.2.  For each $u_0\in L^2(\OO)$, equation \eqref{eq:moving-spde} has a unique weak solution.  There exists $\alpha\in(0,1)$ such that the weak solution $u$ of \eqref{eq:moving-spde} has a modification whose sample paths are H\"older continuous on $Q_{T_0,2T_0}$. Furthermore, for every $p>0$,
\begin{equation}
\E\|u\|_{C^{\alpha/2,\alpha}(Q_{T_0,2T_0})}^p
\leq
C
\left(
\|u_0\|_{L^2(\OO)}
+
K_{0}
\right)^p .
\label{eq:physical-holder-moment}
\end{equation}
Here $C$ depends on $d$, $L$, $\Lambda$, $p,T_0,T,|\OO|$, and on the uniform bounds in Assumption~\ref{ass:moving-domain}.
\end{theorem}

\subsection*{Proof strategy}

The proof has three main steps. First, we transform the equation on moving-domain into an equation on a fixed-domain by means of the diffeomorphism $r(t,\cdot)$. The resulting equation has uniformly elliptic time-dependent coefficients and a bounded first-order term induced by the motion of the domain. The well-posedness of the transformed equation will be proved using the nonhomogeneous variational framework of \cite{PanWangZhaiZhang2026}.
Secondly, we prove the H\"older regularity for the solution of the equation on the fixed-domain. This part combines stochastic De Giorgi iteration, following the approach of \cite{HsuWangWang2017}, with Dirichlet boundary estimates. 
Finally, we transfer the H\"older estimate of the solution of the SPDEs on the fixed domain back to SPDEs the original moving domain. 

\section{Well-posedness }
\label{sec:variational-framework}

In this section, we transfer the stochastic heat equations on moving domains to a SPDE on a bounded domain  and prove the equivalence of the two solutions of SPDEs. Then, we establish the well-posedness using a nonhomogeneous variational framework introduced in \cite{PanWangZhaiZhang2026}. To get an idea how the transformed equation looks like, below we carry out some formal calculations. Rigorous arguments will be presented later.
Define
\begin{equation*}
\widehat{u}(t,y)=u(t,r(t,y)),\qquad y\in\OO .
\end{equation*}
Equivalently, $u(t,x)=\widehat{u}(t,\rho(t,x))$. Put
\begin{equation*}
K_r(t,y):=D_x\rho(t,r(t,y)).
\end{equation*}
Differentiating $\rho(t,r(t,y))=y$ with respect to $y$ gives
\begin{equation*}
D_x\rho(t,r(t,y))D_y r(t,y)=I_d,
\end{equation*}
and hence
\begin{equation}
K_r(t,y)=D_x\rho(t,r(t,y))=(D_y r(t,y))^{-1}.
\label{eq:K-inverse}
\end{equation}
Define
\begin{equation*}
a^{ij}(t,y)
=
\sum_{\ell=1}^d
\pa_{x_\ell}\rho_i(t,r(t,y))\pa_{x_\ell}\rho_j(t,r(t,y)),
\qquad A(t,y)=(a^{ij}(t,y)).
\end{equation*}
Then $A=K_rK_r^\top$, and Assumption~\ref{ass:moving-domain} implies that $A$ is uniformly elliptic.

The identities behind the pullback are as follows:
\begin{equation*}
\pa_{x_\ell}u(t,r(t,y))
=
\pa_{y_i}\widehat{u}(t,y)\pa_{x_\ell}\rho_i(t,r(t,y)),
\end{equation*}
and
\begin{equation*}
\Delta_xu(t,r(t,y))
=
a^{ij}(t,y)\pa_{y_i y_j}\widehat{u}(t,y)
+\Delta_x\rho_i(t,r(t,y))\pa_{y_i}\widehat{u}(t,y).
\end{equation*}
Moreover,
\begin{equation*}
\frac{\partial \widehat{u}}{\partial t}(t,y)
=
\frac{\partial u}{\partial t}(t,r(t,y))
+\nabla_xu(t,r(t,y))\cdot\partial_t r(t,y).
\end{equation*}
Differentiating $\rho(t,r(t,y))=y$ in time yields
\begin{equation*}
\pa_t\rho_i(t,r(t,y))
+
\pa_{x_\ell}\rho_i(t,r(t,y))\pa_t r_\ell(t,y)=0,
\end{equation*}
and therefore
\begin{equation*}
\nabla_xu(t,r(t,y))\cdot\pa_t r(t,y)
=
-\pa_t\rho_i(t,r(t,y))\pa_{y_i}\widehat{u}(t,y), \ \ u(t,x)=\widehat{u}(t,\rho(t,x)).
\end{equation*}
Consequently if $u$ is a solution to equation \eqref{eq:moving-spde}, $\widehat{u}$ formally satisfies the following stochastic partial differential equation on the fixed domain $\OO$:
\begin{align*}
\frac{\partial \widehat{u}}{\partial t}
&=
a^{ij}\pa_{y_i y_j}\widehat{u}
+
\Big[
\Delta_x\rho_i(t,r(t,y))-\pa_t\rho_i(t,r(t,y))
\Big]\pa_{y_i}\widehat{u}
+\widetilde f(t,y,\widehat{u})
+\widetilde g_i(t,y,\widehat{u})\dot w_t^i,
\end{align*}
where
\begin{equation}
\widetilde f(t,y,z)=f(t,r(t,y),z),
\qquad
\widetilde g_i(t,y,z)=g_i(t,r(t,y),z).
\label{eq:fg-pullback}
\end{equation}
Since
\begin{equation*}
a^{ij}\pa_{y_i y_j}\widehat{u}
=
\pa_{y_j}\big(a^{ij}\pa_{y_i}\widehat{u}\big)
-\pa_{y_j}a^{ij}\pa_{y_i}\widehat{u},
\end{equation*}
the same equation may be written in divergence form as
\begin{equation}
\left\{
\begin{aligned}
&\frac{\partial \widehat{u}}{\partial t}
=
\divy(A(t,y)\nabla_y \widehat{u})
+q^i(t,y)\pa_{y_i}\widehat{u}
+\widetilde f(t,y,\widehat{u})
+\widetilde g_i(t,y,\widehat{u})\dot w_t^i,
&&y\in\OO,\ 0<t\leq T,\\
&\widehat{u}(t,y)=0,
&&y\in\partial\OO,\ 0<t\leq T,\\
&\widehat{u}(0,y)=u_0(y),
&&y\in\OO ,
\end{aligned}
\right.
\label{eq:fixed-spde}
\end{equation}
where
\begin{equation}
q^i(t,y)
=
\Delta_x\rho_i(t,r(t,y))
-\pa_t\rho_i(t,r(t,y))
-\pa_{y_j}a^{ij}(t,y).
\label{eq:q-def}
\end{equation}
\vskip 0.4cm
 
Next we will consider the equation (\ref{eq:fixed-spde}) in a nonhomogeneous variational framework. 
Set
\begin{equation*}
H=L^2(\OO),\qquad V=H_0^1(\OO).
\end{equation*}
For $v,w\in H$, define
\begin{equation}
(v,w)_t=\int_{\OO}v(y)w(y)J(t,y)\,dy,
\qquad |v|_t^2=(v,v)_t .
\label{eq:time-inner-product}
\end{equation}
where
$
 J(t,y)=\det D_y r(t,y).
$
Let $H_t$ denote the vector space $H$ endowed with the inner product $(\cdot,\cdot)_t$.
The structural operator $\iota_t^*:H\to H$ and its inverse are given by
\begin{equation}
\iota_t^*v=J(t,\cdot)v,
\qquad
\iota_{-t}^*v=J(t,\cdot)^{-1}v,
\qquad
(v,w)_t=(\iota_t^*v,w)_0.
\label{eq:structural-operators}
\end{equation}

Let $W$ be the $\ell^2$-cylindrical Wiener process generated by the Brownian motions $\{w^i\}_{i\in\N}$. Thus, if $(e_i)_{i\in\N}$ denotes the canonical orthonormal basis of $\ell^2$, then formally
\begin{equation*}
W_t=\sum_{i=1}^{\infty}e_i w_t^i .
\end{equation*}

For $v,\psi\in V$, define the drift $\mathcal A(t,v)\in V^*$ by
\begin{align}
{}_{V^*}\langle \mathcal A(t,v),\psi\rangle_V
&=
-\int_{\OO}A(t,y)\nabla v(y)\cdot\nabla\psi(y)\,dy
\nonumber\\
&\quad
+\int_{\OO}q^i(t,y)\partial_{y_i}v(y)\psi(y)\,dy
+\int_{\OO}\widetilde f(t,y,v(y))\psi(y)\,dy.
\label{eq:A-operator}
\end{align}
The diffusion operator $\mathcal B(t,v)\in L_2(\ell^2,H)$ is defined by
\begin{equation}
[\mathcal B(t,v)h](y)
=
\sum_{i=1}^{\infty}\widetilde g_i(t,y,v(y))h_i,
\qquad h=(h_i)_{i\in\N}\in\ell^2,
\label{eq:B-operator}
\end{equation}
or, equivalently,
\begin{equation*}
(\mathcal B(t,v)e_i,\psi)_0
=
\int_{\OO}\widetilde g_i(t,y,v(y))\psi(y)\,dy.
\end{equation*}

Consider the following evolution equation in Gelfand triplet $V\subset H\subset V^*$: 

\begin{equation}
dY(t)=\mathcal A(t,Y(t))\,dt+\mathcal B(t,Y(t))\,dW_t,
\qquad Y(0)=u_0.
\label{eq:abstract-fixed-equation}
\end{equation}

\begin{definition}
\label{def:fixed-variational-solution}

An adapted process $Y$ is called a variational solution of
\eqref{eq:abstract-fixed-equation} if
\begin{equation*}
Y\in L^2(\Omega;C([0,T];H))\cap L^2(\Omega;L^2(0,T;V)),
\end{equation*}
and, for every deterministic $\phi\in H^1([0,T];V)$ and every
$t\in[0,T]$, the following identity holds almost surely:
\begin{align}
(Y(t),\phi(t))_0
&=
(u_0,\phi(0))_0
+\int_0^t{}_{V^*}\langle \mathcal A(s,Y(s)),\phi(s)\rangle_V\,ds
\nonumber\\
&\quad
+\int_0^t(\mathcal B(s,Y(s))\,dW_s,\phi(s))_0
+\int_0^t(Y(s),\dot\phi(s))_0\,ds.
\label{eq:fixed-variational-form}
\end{align}
Here $\dot\phi$ denotes the weak time derivative of $\phi$.

\end{definition}

Now we can state the main result of Section \ref{sec:variational-framework}.
\begin{theorem}
\label{thm:moving-wellposed}
Under Assumptions~\ref{ass:moving-domain} and \ref{ass:coefficients}, and for every $u_0\in L^2(\Omega,\mathcal F_0;L^2(\OO))$, the equation \eqref{eq:moving-spde} has a unique weak solution in the sense of Definition~\ref{def:moving-solution}. 
\end{theorem}

\begin{proof}
We will invoke Theorem 2.8 from \cite{PanWangZhaiZhang2026}  in a nonhomogeneous variational framework. The proof proceeds in three steps, presented in the subsections below. First, in Section~\ref{sec:solution-equivalence}, we establish the equivalence between weak solutions of \eqref{eq:moving-spde} and variational solutions of \eqref{eq:abstract-fixed-equation}. Next, in Section~\ref{Verification of C1--C4}, we verify the conditions C1--C4 required in in \cite{PanWangZhaiZhang2026} for the associated time-dependent Hilbert-space structure. In Section~\ref{Verification of H1--H5}, we verify the variational hypotheses H1--H5 for the drift and diffusion operators required in \cite{PanWangZhaiZhang2026}. Theorem 2.8 in \cite{PanWangZhaiZhang2026} then yields the existence and uniqueness of a variational solution to equation \eqref{eq:abstract-fixed-equation}, and the equivalence established in Section~\ref{sec:solution-equivalence} transfers this result to equation \eqref{eq:moving-spde}.
\end{proof}
\subsection{Equivalence of the two notions of solution}
\label{sec:solution-equivalence}

In this subsection, we prove that the two solution concepts introduced in
Definitions~\ref{def:moving-solution} and
\ref{def:fixed-variational-solution} are equivalent under the
change of variables induced by the flow $r$. More precisely, we show that
the pullback transformation
\begin{equation}
Y(t,y)=u(t,r(t,y)),
\qquad (t,y)\in[0,T]\times\OO,
\label{eq:solution-correspondence}
\end{equation}
maps every weak solution $u$ of the problem
\eqref{eq:moving-spde} on moving-domain to a variational solution $Y$ of the problem \eqref{eq:abstract-fixed-equation} on fixed-domain. Conversely, using the inverse
flow $\rho$, we define
\begin{equation}
u(t,x)=Y(t,\rho(t,x)),
\qquad (t,x)\in[0,T]\times\overline{\mathcal O_t},
\label{eq:inverse-solution-correspondence}
\end{equation}
and prove that every variational solution $Y$ gives rise to a weak solution
$u$ on the moving domains. 

Recall $q^i$ was defined in \eqref{eq:q-def}. The following lemma establishes the properties of $q^i$ needed for the analysis of the transformed equation.
\begin{lemma}\label{lem:weak-flow-geometric-identities}
Fix $t\in[0,T]$. The following identities hold in the sense of
distributions on $\OO$:
\begin{align}
&J(t,y)\Delta_x\rho_i(t,r(t,y))
=
\partial_{y_j}\bigl(J(t,y)a^{ij}(t,y)\bigr),
\label{eq:rho-laplacian-coordinate-identity}\\
&\partial_tJ(t,y)
=-
\partial_{y_i}\Bigl(
J(t,y)\partial_t\rho_i(t,r(t,y))
\Bigr),
\label{eq:jacobian-transport-equivalence}\\
&J(t,y)q^i(t,y)
=
a^{ij}(t,y)\partial_{y_j}J(t,y)
-
J(t,y)\partial_t\rho_i(t,r(t,y)).
\label{eq:q-geometric-form}
\end{align}
\end{lemma}

\begin{proof}
For any $\zeta\in C_c^\infty(\OO)$. By the change of variables, we have
\begin{align*}
\int_{\OO}
J(t,y)\Delta_x\rho_i(t,r(t,y))\zeta(y)\,dy
&=
\int_{\mathcal O_t}
\Delta_x\rho_i(t,x)\zeta(\rho(t,x))\,dx\\
&=
-\int_{\mathcal O_t}
\partial_{x_\ell}\rho_i(t,x)
\partial_{x_\ell}\!\left[\zeta(\rho(t,x))\right]\,dx\\
&=
-\int_{\mathcal O_t}
\partial_{x_\ell}\rho_i(t,x)
\partial_{x_\ell}\rho_j(t,x)
\partial_{y_j}\zeta(\rho(t,x))\,dx\\
&=
-\int_{\OO}
J(t,y)a^{ij}(t,y)\partial_{y_j}\zeta(y)\,dy.
\end{align*}
This proves \eqref{eq:rho-laplacian-coordinate-identity} in
$\mathcal D'(\OO)$.

Define
$
K(t,y):=(D_y r(t,y))^{-1}.
$
Differentiating
$
\rho(t,r(t,y))=y
$
with respect to $t$ gives 
\begin{equation}
\partial_t\rho_i(t,r(t,y))
=
-K_{ik}(t,y)\partial_t r_k(t,y).
\label{eq:inverse-flow-time-identity}
\end{equation}
Using \eqref{eq:inverse-flow-time-identity} together with the fact that $\partial_tJ(t,y)
=
J(t,y)K_{ik}(t,y)
\partial_{y_i}\partial_t r_k(t,y)$, we obtain
\begin{align}
&\int_{\OO}\partial_tJ(t,y)\xi(y)\,dy
-
\int_{\OO}
J(t,y)\partial_t\rho_i(t,r(t,y))
\partial_{y_i}\xi(y)\,dy
\nonumber\\
=&\int_{\OO}J(t,y)K_{ik}(t,y)
\partial_{y_i}\partial_t r_k(t,y)\xi(y)\,dy+\int_{\OO}
J(t,y)K_{ik}(t,y)\partial_t r_k(t,y)
\partial_{y_i}\xi(y)\,dy\nonumber\\
=&
\int_{\OO}
J(t,y)K_{ik}(t,y)
\partial_{y_i}\Bigl(
\partial_t r_k(t,y)\xi(y)
\Bigr)\,dy.
\label{eq:transport-weak-step-three}
\end{align}
Since
$
J(t,y)K_{ik}(t,y)
$
is the corresponding entry of the cofactor matrix of $D_y r(t,y)$,
the cofactor-divergence identity gives
\begin{equation}
\partial_{y_i}\bigl(J(t,y)K_{ik}(t,y)\bigr)=0.
\label{eq:cofactor-divergence-identity}
\end{equation}
Therefore, integration by parts in
\eqref{eq:transport-weak-step-three} yields
\begin{align*}
&\int_{\OO}\partial_tJ(t,y)\xi(y)\,dy
-
\int_{\OO}
J(t,y)\partial_t\rho_i(t,r(t,y))
\partial_{y_i}\xi(y)\,dy
=0.
\end{align*}
This is precisely
\eqref{eq:jacobian-transport-equivalence}.

Finally, we prove \eqref{eq:q-geometric-form}. By
\eqref{eq:q-def}, for every $\xi\in C_c^\infty(\OO)$,
\begin{align*}
\int_{\OO}J(t,y)q^i(t,y)\xi(y)\,dy
=&
\int_{\OO}
J(t,y)\Delta_x\rho_i(t,r(t,y))\xi(y)\,dy
-\int_{\OO}
J(t,y)\partial_t\rho_i(t,r(t,y))\xi(y)\,dy\\
&-\int_{\OO}
J(t,y)\partial_{y_j}a^{ij}(t,y)\xi(y)\,dy.
\end{align*}
Using \eqref{eq:rho-laplacian-coordinate-identity}, the first term on
the right-hand side satisfies
\begin{align*}
&\int_{\OO}
J(t,y)\Delta_x\rho_i(t,r(t,y))\xi(y)\,dy
=
\left\langle
\partial_{y_j}\bigl(J(t,\cdot)a^{ij}(t,\cdot)\bigr),
\xi
\right\rangle.
\end{align*}
Combining the equations above yields 
\eqref{eq:q-geometric-form}.
\end{proof}

We now describe the corresponding transformation for test functions. Let
$\varphi$ be a deterministic test function on the moving cylinder and define
its pullback by
\begin{equation}
\widehat{\varphi}(t,y):=\varphi(t,r(t,y)).
\label{eq:test-pullback}
\end{equation}
Applying the chain rule to $\varphi$ gives
\begin{align}
\nabla_x\varphi(t,r(t,y))
 &=K_r(t,y)^\top\nabla_y\widehat{\varphi}(t,y),
\label{eq:test-gradient-transform}\\
\partial_t\varphi(t,r(t,y))
 &=\partial_t\widehat{\varphi}(t,y)
 +\partial_t\rho_i(t,r(t,y))\partial_{y_i}\widehat{\varphi}(t,y).
\label{eq:test-time-transform}
\end{align}

For the pulled-back solution $Y(t,y)=u(t,r(t,y))$ and the pulled-back test
function $\widehat{\varphi}(t,y)=\varphi(t,r(t,y))$, the spatial bilinear
forms are  transformed as
\begin{equation}
\int_{\mathcal O_t}\nabla_xu(t,x)\cdot\nabla_x\varphi(t,x)\,dx
=\int_{\OO}J(t,y)A(t,y)\nabla_yY(t,y)\cdot
\nabla_y\widehat{\varphi}(t,y)\,dy.
\label{eq:weak-gradient-change}
\end{equation}
This identity is first verified for smooth $u$ and $\varphi$ and then
extended to $H_0^1$ by density.

Below is the  equivalence of the two notions of solutions.
\begin{proposition}
\label{prop:solution-equivalence}
Under Assumptions 2.1 and 2.2, a process $u$ is a weak solution of the equation \eqref{eq:moving-spde} in the sense of Definition~\ref{def:moving-solution} if and only if its pullback transformation
$
Y(t,y)=u(t,r(t,y))
$
is a variational solution of the equation \eqref{eq:abstract-fixed-equation}. 
\end{proposition}

\begin{proof}
Let $u$ be the weak solution of \eqref{eq:moving-spde}.
Recall that
\[
H=L^2(\OO),\qquad V=H_0^1(\OO).
\]
Let $Y(t,y)=u(t,r(t,y))$. By the change of variables and the uniform $C^1$-regularity of $r$, we have
\begin{equation*}
Y\in L^2(\Omega;C([0,T];H))
\cap L^2(\Omega;L^2(0,T;V)).
\end{equation*}
For any
$\phi\in C^2([0,T];C_c^\infty(\OO))$, define a test function on the moving domain by
\begin{equation}
\varephi(t,x)
:=
\frac{\phi(t,\rho(t,x))}
{J(t,\rho(t,x))}.
\label{eq:moving-test-from-fixed}
\end{equation}
Then $\widehat\varphi(t,y)=\phi(t,y)/J(t,y)$. By the change of variables, we have
\begin{equation}
\int_{\mathcal O_t}u(t,x)\varephi(t,x)\,dx
=\int_{\OO}Y(t,y)\phi(t,y)\,dy
=(Y(t),\phi(t))_0.
\label{eq:pairing-transform-fixed}
\end{equation}
Note that
\begin{align}
\int_{\mathcal O_s}f(s,x,u(s,x))\varphi(s,x)\,dx
&=\int_{\OO}\widetilde f(s,y,Y(s,y))\phi(s,y)\,dy,
\label{eq:reaction-transform-fixed}\\
\int_{\mathcal O_s}g_i(s,x,u(s,x))\varphi(s,x)\,dx
&=\int_{\OO}\widetilde g_i(s,y,Y(s,y))\phi(s,y)\,dy.
\label{eq:noise-transform-fixed}
\end{align}
We claim that
\begin{align}
&-\int_{\mathcal O_s}
\nabla_xu(s,x)\cdot\nabla_x\varphi(s,x)\,dx
+\int_{\mathcal O_s}
u(s,x)\partial_s\varphi(s,x)\,dx\nonumber\\
&\qquad=
{}_{V^*}\langle\mathcal A(s,Y(s)),\phi(s)\rangle_V
-\int_{\OO}\widetilde f(s,y,Y(s,y))\phi(s,y)\,dy
+(Y,\dot\phi)_0.
\label{eq:drift-transport-equivalence}
\end{align}
In view of \eqref{eq:test-gradient-transform}, the right-hand side of the preceding equality can be written as
\begin{align*}
&-\int_{\OO}J(s,y)A(s,y)\nabla_yY(s,y)
\cdot\nabla_y\widehat{\varphi}(s,y)\,dy
+\int_{\OO}J(s,y)Y(s,y)
\left[
\partial_s\widehat{\varphi}(s,y)
+\partial_s\rho_i(s,r(s,y))
 \partial_{y_i}\widehat{\varphi}(s,y)
\right]\,dy\\
&=
-\int_{\OO}A(s,y)\nabla_yY(s,y)
\cdot\nabla_y\phi(s,y)\,dy
+\int_{\OO}a^{ij}(s,y)\partial_{y_i}Y(s,y)
\,\phi(s,y)
\frac{\partial_{y_j}J(s,y)}{J(s,y)}\,dy\\
&\quad
+\int_{\OO}Y(s,y)\partial_s\phi(s,y)\,dy
+\int_{\OO}Y(s,y)
\partial_s\rho_i(s,r(s,y))
\partial_{y_i}\phi(s,y)\,dy\\
&\quad
-\int_{\OO}Y(s,y)
\partial_s\rho_i(s,r(s,y))
\phi(s,y)
\frac{\partial_{y_i}J(s,y)}{J(s,y)}\,dy
-\int_{\OO}Y(s,y)\phi(s,y)
\frac{\partial_sJ(s,y)}{J(s,y)}\,dy.
\end{align*}
Combining the above equation with \eqref{eq:jacobian-transport-equivalence} and \eqref{eq:q-geometric-form}, we obtain \eqref{eq:drift-transport-equivalence}.

Substituting \eqref{eq:pairing-transform-fixed}--
\eqref{eq:drift-transport-equivalence} into
\eqref{eq:moving-weak-form} yields that \eqref{eq:fixed-variational-form} holds for $\phi\in C^2([0,T];C_c^\infty(\OO))$. Note that $
C^2([0,T];C_c^\infty(\OO))
$
is dense in $H^1(0,T;V)$. The identity  \eqref{eq:fixed-variational-form} for every
$\phi\in H^1([0,T];V)$ follows by a standard approximation argument.
Thus $Y$ is a variational solution in the sense of
Definition~\ref{def:fixed-variational-solution}.

Conversely, let $Y$ be a variational solution of
\eqref{eq:abstract-fixed-equation} and define $u(t,x)=Y(t,\rho(t,x))$. For every $t\in[0,T]$,
the change of variables gives
\[
\|u(t)\|_{L^2(\mathcal O_t)}\leq C\|Y(t)\|_H.
\]
The uniform $C^1$ bounds on $\rho$ and the
Jacobian bounds imply
\[
\|u(t)\|_{H_0^1(\mathcal O_t)}
\leq C\|Y(t)\|_V.
\]
Hence $u\in L^2(\Omega;C([0,T];L^2(\mathcal O_{\cdot})))
\cap
L^2(\Omega;L^2(0,T;H_0^1(\mathcal O_{\cdot})))$. 

Let $\varephi\in C_D^2(Q_{T})$ and set
\begin{equation}
\phi(t,y):=J(t,y)\widehat\varephi(t,y).
\label{eq:fixed-test-from-moving}
\end{equation}
Then $\phi\in H^1([0,T];V)$. By \eqref{eq:fixed-variational-form}, we have
\begin{align*}
(Y(t),\phi(t))_0
&=
(u_0,\phi(0))_0
+\int_0^t{}_{V^*}\langle \mathcal A(s,Y(s)),\phi(s)\rangle_V\,ds
\nonumber\\
&\quad
+\int_0^t(\mathcal B(s,Y(s))\,dW_s,\phi(s))_0
+\int_0^t(Y(s),\dot\phi(s))_0\,ds.
\end{align*}
Note that
\begin{align}
(Y(t),\phi(t))_0
=\int_{\mathcal O_t}u(t,x)\varephi(t,x)\,dx,\ \
(u_0,\phi(0))_0
=\int_{\OO}u_0(x)\varephi(0,x)\,dx.
\label{eq:reverse-terminal-pairing}
\end{align}
And the stochastic integral term satisfies
\begin{align}
\int_0^t(\mathcal B(s,Y(s))\,dW_s,\phi(s))_0
&=\sum_{i=1}^{\infty}
\int_{\OO}J(s,y)\widetilde g_i(s,y,Y(s,y))
\widehat{\varphi}(s,y)\,dy\,dw_s^i
\nonumber\\
&=\sum_{i=1}^{\infty}
\int_{\mathcal O_s}g_i(s,x,u(s,x))\varphi(s,x)\,dx\,dw_s^i.
\label{eq:reverse-noise-change}
\end{align}
We claim that
\begin{align}
{}_{V^*}\langle\mathcal A(s,Y),\phi\rangle_V
+(Y,\dot\phi)_0
=&-\int_{\mathcal O_s}\nabla_xu(s,x)\cdot\nabla_x\varphi(s,x)\,dx
+\int_{\mathcal O_s}u(s,x)\,\partial_s\varphi(s,x)\,dx\nonumber\\
&+\int_{\mathcal O_s}f(s,x,u(s,x))\varphi(s,x)\,dx.
\label{eq:reverse-drift-time-transform}
\end{align}
By \eqref{eq:fixed-test-from-moving} and the definition of
$\mathcal A$, we have
\begin{align}
{}_{V^*}\langle\mathcal A(s,Y),\phi\rangle_V
&=
-\int_{\OO}A(s,y)\nabla_yY(s,y)\cdot
\nabla_y\!\bigl(J(s,y)\widehat{\varphi}(s,y)\bigr)\,dy
\nonumber\\
&\quad
+\int_{\OO}J(s,y)q^i(s,y)
\partial_{y_i}Y(s,y)\widehat{\varphi}(s,y)\,dy
\nonumber\\
&\quad
+\int_{\OO}J(s,y)\widetilde f(s,y,Y(s,y))
\widehat{\varphi}(s,y)\,dy\nonumber\\
&=
-\int_{\OO}J(s,y)A(s,y)\nabla_yY(s,y)\cdot
\nabla_y\widehat{\varphi}(s,y)\,dy
\nonumber\\
&\quad
-\int_{\OO}J(s,y)\partial_s\rho_i(s,r(s,y))
\partial_{y_i}Y(s,y)\widehat{\varphi}(s,y)\,dy
\nonumber\\
&\quad
+\int_{\OO}J(s,y)\widetilde f(s,y,Y(s,y))
\widehat{\varphi}(s,y)\,dy,
\label{eq:reverse-principal-transport}
\end{align}
where the last step is due to \eqref{eq:q-geometric-form}.
Note that
\begin{align}
(Y,\dot\phi)_0
&=
\int_{\OO}J(s,y)Y(s,y)
\partial_s\widehat{\varphi}(s,y)\,dy
+\int_{\OO}Y(s,y)(\partial_sJ(s,y))
\widehat{\varphi}(s,y)\,dy.
\label{eq:reverse-time-term-expansion}
\end{align}
By adding \eqref{eq:reverse-principal-transport} and \eqref{eq:reverse-time-term-expansion}, and applying \eqref{eq:jacobian-transport-equivalence} together with \eqref{eq:weak-gradient-change}, we arrive at \eqref{eq:reverse-drift-time-transform}.

Substituting \eqref{eq:reverse-terminal-pairing}--
\eqref{eq:reverse-drift-time-transform} into
\eqref{eq:fixed-variational-form} gives 
\eqref{eq:moving-weak-form}. Therefore $u$ is a weak solution on moving domains
in the sense of Definition~\ref{def:moving-solution}, and the two notions of
solution are equivalent.
\end{proof}

\subsection{Verification of C1--C4.}\label{Verification of C1--C4}
In this subsection, we verify conditions {\rm(C1)--(C4)} required by the
nonhomogeneous variational framework of
\cite[Section~3.4]{PanWangZhaiZhang2026} for the time-dependent Hilbert
structure introduced in (\ref{eq:time-inner-product}) and (\ref{eq:structural-operators}). Since $r(0,y)=y$, we have $J(0,y)=1$ and hence
the inner product $(\cdot,\cdot)_0$ coincides with the standard inner
product on $H=L^2(\OO)$. We first recall the four structural conditions required in \cite{PanWangZhaiZhang2026} and
then verify them one by one.
\begin{enumerate}[label=(C\arabic*)]
\item there exists $c_1\geq1$ such that
\begin{equation*}
c_1^{-1}|v|_0\leq |v|_t\leq c_1|v|_0,
\qquad v\in H,\quad t\in[0,T];
\end{equation*}
\item there is a family of self-adjoint operators
$\{\Phi(t)\}_{t\in[0,T]}\subset\mathcal L(H)$ such that
\begin{equation*}
\int_0^{T}\|\Phi(t)\|_{\mathcal L(H)}\,dt<\infty
\end{equation*}
and
\begin{equation}
|v|_t^2-|v|_0^2
=\int_0^t(v,\Phi(s)v)_0\,ds,
\qquad v\in H,\quad t\in[0,T];
\label{eq:C2-identity}
\end{equation}
\item $\iota_t^*$ maps $V$ into itself and
\begin{equation*}
\|\iota_t^*v\|_V\leq c_2\|v\|_V,
\qquad v\in V,\quad t\in[0,T],
\end{equation*}
for a constant $c_2$ independent of $t$;
\item $\iota_t^*:V\to V$ is bijective and
\begin{equation*}
\|\iota_{-t}^*v\|_V\leq c_3\|v\|_V,
\qquad v\in V,\quad t\in[0,T],
\end{equation*}
for a constant $c_3$ independent of $t$.
\end{enumerate}

\emph{Verification of C1.}
Put
\begin{equation*}
m:=\inf_{(t,y)\in[0,T]\times\OO}J(t,y),
\qquad
M:=\sup_{(t,y)\in[0,T]\times\OO}J(t,y).
\end{equation*}
Assumption~\ref{ass:moving-domain} implies $0<m\leq M<\infty$.
Therefore, for every $v\in H$ and $t\in[0,T]$,
\begin{equation}
m|v|_0^2
\leq
|v|_t^2
=\int_{\OO}|v(y)|^2J(t,y)\,dy
\leq
M|v|_0^2.
\label{eq:C1-equivalence}
\end{equation}
Thus C1 holds with
$c_1=\max\{m^{-1/2},M^{1/2}\}$.

\emph{Verification of C2.}
Differentiating $J(t,y)=
\det D_y r(t,y)$ with respect to the time variable, we have
\begin{align}
\partial_tJ(t,y)
&=
\partial_t\det D_y r(t,y)
\nonumber\\
&=
\det D_y r(t,y)
\operatorname{tr}
\left(
(D_y r(t,y))^{-1}D_y\partial_t r(t,y)
\right)
\nonumber\\
&=
J(t,y)
\operatorname{tr}
\left(
(D_y r(t,y))^{-1}D_y\partial_t r(t,y)
\right).
\label{eq:jacobi-formula-J}
\end{align}
By Assumption~\ref{ass:moving-domain}, 
\begin{equation}
\sup_{t\in[0,T]}
\|\partial_tJ(t,\cdot)\|_{L^\infty(\OO)}
<\infty.
\label{eq:J-time-bound}
\end{equation}

We now define, for every $t\in[0,T]$, the multiplication operator
$\Phi(t):H\to H$ by
\begin{equation}
[\Phi(t)v](y)
:=
\partial_tJ(t,y)v(y),
\qquad v\in H=L^2(\OO).
\label{eq:C2-Phi}
\end{equation}
For every $v,w\in H$,
\begin{align*}
(\Phi(t)v,w)_0
&=
\int_{\OO}\partial_tJ(t,y)v(y)w(y)\,dy\\
&=
\int_{\OO}v(y)\partial_tJ(t,y)w(y)\,dy
=
(v,\Phi(t)w)_0.
\end{align*}
Furthermore, for every $v\in H$,
\begin{align*}
\|\Phi(t)v\|_H^2
&=
\int_{\OO}
|\partial_tJ(t,y)|^2|v(y)|^2\,dy\\
&\leq
\|\partial_tJ(t,\cdot)\|_{L^\infty(\OO)}^2
\|v\|_H^2.
\end{align*}
It follows that
\begin{equation*}
\|\Phi(t)\|_{\mathcal L(H)}
\leq
\|\partial_tJ(t,\cdot)\|_{L^\infty(\OO)}<\infty.
\end{equation*}
Since $J(0,y)=1$, we have for every $v\in H$ 
\begin{align}
|v|_t^2-|v|_0^2
&=
\int_{\OO}|v(y)|^2\bigl(J(t,y)-J(0,y)\bigr)\,dy
\nonumber\\
&=
\int_0^t\int_{\OO}
|v(y)|^2\partial_sJ(s,y)\,dy\,ds
\nonumber\\
&=
\int_0^t(v,\Phi(s)v)_0\,ds.
\end{align}
Thus we confirm the validity of C2.


\emph{Verification of C3.}
By Assumption \ref{ass:moving-domain}, we have
\begin{equation}
\sup_{t\in[0,T]}
\|J(t,\cdot)\|_{W^{1,\infty}(\OO)}<\infty.
\label{eq:J-W1infty}
\end{equation}
Recalling $\iota_t^*v=J(t,\cdot)v$, the product rule gives, for every
$v\in V$,
\begin{align*}
\|\iota_t^*v\|_V
=\|J(t,\cdot)v\|_{H_0^1(\OO)}
&\leq
\|J(t,\cdot)\|_{L^\infty(\OO)}
\|v\|_{H_0^1(\OO)}
+
\|\nabla_yJ(t,\cdot)\|_{L^\infty(\OO)}
\|v\|_{L^2(\OO)}\\
&\leq C\|v\|_V,
\end{align*}
where $C$ is independent of $t$. Hence, condition {\rm(C3)} holds.

\emph{Verification of C4.}
By Assumption \ref{ass:moving-domain}, there exists a constant $\kappa>0$ such that for any $t\in[0,T]$,
\[
\|J(t,\cdot)^{-1}\|_{L^\infty(\OO)}
\leq \kappa^{-1}.
\]
Moreover, the chain rule gives
\[
\nabla_y\bigl(J(t,y)^{-1}\bigr)
=
-\frac{\nabla_yJ(t,y)}{J(t,y)^2}.
\]
Consequently,
\[
\|\nabla_yJ(t,\cdot)^{-1}\|_{L^\infty(\OO)}
\leq
\kappa^{-2}\|\nabla_yJ(t,\cdot)\|_{L^\infty(\OO)}.
\]
Together with \eqref{eq:J-W1infty}, this yields
\begin{equation}
\sup_{t\in[0,T]}
\|J(t,\cdot)^{-1}\|_{W^{1,\infty}(\OO)}
<\infty.
\label{eq:J-inverse-W1infty}
\end{equation}
Since $\iota_{-t}^*v=J(t,\cdot)^{-1}v$, the product rule and
\eqref{eq:J-inverse-W1infty} imply
\begin{equation*}
\|\iota_{-t}^*v\|_V
=
\|J(t,\cdot)^{-1}v\|_{H_0^1(\OO)}
\leq C\|v\|_V,
\qquad v\in V,
\end{equation*}
with $C$ independent of $t$. Finally,
\[
\iota_t^*\iota_{-t}^*
=
\iota_{-t}^*\iota_t^*
=
I
\quad\text{on }V.
\]
Thus $\iota_t^*:V\to V$ is bijective, and condition {\rm(C4)} follows.
This completes the verification of C1--C4.

\subsection{Verification of H1--H5.}\label{Verification of H1--H5}
In this subsection, we verify hypotheses {\rm(H1)--(H5)} in the
nonhomogeneous variational framework of
\cite[Section~2.3]{PanWangZhaiZhang2026}. We first clarify the action of
$\iota_t^*$ on the dual space $V^*$. 
In particular, for the drift operator $\mathcal A$,
\begin{equation*}
{}_{V^*}\langle\iota_t^*\mathcal A(t,v),\psi\rangle_V
=
{}_{V^*}\langle\mathcal A(t,v),J(t,\cdot)\psi\rangle_V.
\end{equation*}

For
$
h\in L^1(0,T;\mathbb R_+)
$, and 
$
\alpha\in(1,\infty),
$
consider the following hypotheses stated in \cite{PanWangZhaiZhang2026}. 
\begin{enumerate}[label=(H\arabic*)]
\item For almost every $t\in[0,T]$ and every $u,v,x\in V$, the map
\begin{equation*}
\lambda\in\R\longmapsto
{}_{V^*}\langle\mathcal A(t,u+\lambda v),\iota_t^*x\rangle_V
\end{equation*}
is continuous.

\item There exist a locally bounded measurable function
$\varrho:V\to[0,\infty)$ and constants $\gamma,C\geq0$ such that, for
almost every $t$ and all $u,v\in V$,
\begin{align}
&2{}_{V^*}\langle
\iota_t^*\mathcal A(t,u)-\iota_t^*\mathcal A(t,v),u-v
\rangle_V
+\|\mathcal B(t,u)-\mathcal B(t,v)\|_{L_2(\ell^2,H_t)}^2
\nonumber\\
&\qquad\leq
\bigl[h(t)+\varrho(v)\bigr]|u-v|_t^2,
\label{eq:H2-abstract}
\end{align}
and
\begin{equation*}
|\varrho(u)|
\leq C(1+\|u\|_V^\alpha)(1+|u|_0^\gamma).
\end{equation*}

\item There exists $c>0$ such that, for almost every $t$ and every
$u\in V$,
\begin{equation}
2{}_{V^*}\langle\iota_t^*\mathcal A(t,u),u\rangle_V
+\|\mathcal B(t,u)\|_{L_2(\ell^2,H_t)}^2
\leq h(t)(1+|u|_t^2)-c\|u\|_V^\alpha.
\label{eq:H3-abstract}
\end{equation}

\item There exist constants $\beta,C\geq0$ such that, for almost every
$t$ and every $u\in V$,
\begin{equation}
\|\iota_t^*\mathcal A(t,u)\|_{V^*}^{\frac{\alpha}{\alpha-1}}
\leq
h(t)+C\|u\|_V^\alpha(1+|u|_0^\beta).
\label{eq:H4-abstract}
\end{equation}

\item For almost every $t$ and every $u\in V$,
\begin{equation}
\|\mathcal B(t,u)\|_{L_2(\ell^2,H_t)}^2
\leq h(t)(1+|u|_t^2).
\label{eq:H5-abstract}
\end{equation}
\end{enumerate}
We verify these hypotheses for equation \eqref{eq:abstract-fixed-equation} with
\begin{equation}
\alpha=2,\qquad \beta=\gamma=0,\qquad
\varrho(v)\equiv C_0,\qquad
h(t)\equiv C_0
\bigl(1+ K_{0}^2\bigr),
\label{eq:H-parameter-choice}
\end{equation}
where $C_0$ may be enlarged below but is independent of $t$ and $v$.

\emph{Verification of H1.}
For $u,v,x\in V$,
\begin{align*}
&{}_{V^*}\langle
\mathcal A(t,u+\lambda v),\iota_t^*x\rangle_V
\nonumber\\
=&
-\int_{\OO}A(t,y)\nabla_y(u+\lambda v)(y)\cdot\nabla_y(J(t,y)x(y))\,dy
+\int_{\OO}q^i(t,y)\partial_{y_i}(u+\lambda v)(y)J(t,y)x(y)\,dy
\nonumber\\
&
+\int_{\OO}\widetilde f(t,y,u(y)+\lambda v(y))J(t,y)x(y)\,dy.
\end{align*}
The first two terms are affine functions of $\lambda$ and hence depend
continuously on $\lambda$.
For the last term, by \eqref{eq:fg-lipschitz} we have when
$\lambda_n\to\lambda$,
\begin{equation*}
\|\widetilde f(t,\cdot,u(\cdot)+\lambda_n v(\cdot))
-\widetilde f(t,\cdot,u(\cdot)+\lambda v(\cdot))\|_{L^2(\OO)}
\leq L|\lambda_n-\lambda|\,\|v\|_{L^2(\OO)}\to0.
\end{equation*}
Since $J(t,\cdot)x\in L^2(\OO)$, the last term is also continuous. Thus H1
holds.

\emph{Verification of H2.}
Set $z=u-v$. By the uniform ellipticity of $A$, there exists a constant
$\lambda_A>0$ such that
\begin{equation*}
\xi^\top A(t,y)\xi
\geq
\lambda_A|\xi|^2,
\qquad
\xi\in\mathbb R^d,
\end{equation*}
for almost every $(t,y)\in[0,T]\times\OO$. 
By \eqref{eq:jacobian-transport-equivalence}, \eqref{eq:q-geometric-form}, we obtain
\begin{align}
&2{}_{V^*}\left\langle
\iota_t^*\mathcal A(t,u)-\iota_t^*\mathcal A(t,v),z
\right\rangle_V
\nonumber\\
&\leq
-2\lambda_A\int_{\OO}J(t,y)
 |\nabla_y z(y)|^2\,dy
-2\int_{\OO}J(t,y)
 \partial_t\rho_i(t,r(t,y))
 \partial_{y_i}z(y)z(y)\,dy
\nonumber\\
&\quad
+2\int_{\OO}J(t,y)
 \bigl(\widetilde f(t,y,u(y))
      -\widetilde f(t,y,v(y))\bigr)z(y)\,dy
\nonumber\\
&=
-2\lambda_A\int_{\OO}J(t,y)
 |\nabla_y z(y)|^2\,dy
+\int_{\OO}\partial_{y_i}
 \Bigl(J(t,y)\partial_t\rho_i(t,r(t,y))\Bigr)
 z(y)^2\,dy
\nonumber\\
&\quad
+2\int_{\OO}J(t,y)
 \bigl(\widetilde f(t,y,u(y))
      -\widetilde f(t,y,v(y))\bigr)z(y)\,dy
\nonumber\\
&=
-2\lambda_A\int_{\OO}J(t,y)
 |\nabla_y z(y)|^2\,dy
-\int_{\OO}\partial_tJ(t,y)z(y)^2\,dy
\nonumber\\
&\quad
+2\int_{\OO}J(t,y)
 \bigl(\widetilde f(t,y,u(y))
      -\widetilde f(t,y,v(y))\bigr)z(y)\,dy
\nonumber\\
&\leq
-2\lambda_A \kappa \|z\|_V^2
+C_f|z|_t^2.
\label{eq:weighted-drift-monotonicity}
\end{align}
Here we used the boundedness of $\partial_tJ$ and the Lipschitz continuity
of $\widetilde f$ with respect to its last variable.

Moreover, by the Lipschitz continuity of $\widetilde g$ with respect to its
last variable, there exists a constant $L_g>0$ such that
\begin{equation*}
\|\mathcal B(t,u)-\mathcal B(t,v)\|_{L_2(\ell^2,H_t)}^2
 =\int_{\OO}J(t,y)
\|\widetilde g(t,y,u(y))-\widetilde g(t,y,v(y))\|_{\ell^2}^2\,dy
\leq L_g|u-v|_t^2.
\end{equation*}
Combining this estimate with \eqref{eq:weighted-drift-monotonicity}, we obtain
\begin{align}
&2{}_{V^*}\langle
\iota_t^*\mathcal A(t,u)-\iota_t^*\mathcal A(t,v),u-v
\rangle_V
+\|\mathcal B(t,u)-\mathcal B(t,v)\|_{L_2(\ell^2,H_t)}^2
\nonumber\\
&\qquad\leq
-2\lambda_AJ_*\|u-v\|_V^2
+\bigl(C_f+L_g\bigr)|u-v|_t^2.
\label{eq:H2-corrected}
\end{align}
Thus, condition \eqref{eq:H2-abstract} holds with $
\varrho=C_f+L_g$.

\emph{Verification of H3.}
Using \eqref{eq:q-geometric-form}, we obtain
\begin{align}
&2{}_{V^*}\langle\iota_t^*\mathcal A(t,v),v\rangle_V
+\|\mathcal B(t,v)\|_{L_2(\ell^2,H_t)}^2
\nonumber\\
&=
-2\int_{\OO}J(t,y)A(t,y)\nabla_yv(y)\cdot\nabla_yv(y)\,dy
-2\int_{\OO}J(t,y)\partial_t\rho_i(t,r(t,y))
 \partial_{y_i}v(y)v(y)\,dy
\nonumber\\
&\quad
+2\int_{\OO}J(t,y)\widetilde f(t,y,v(y))v(y)\,dy
+\int_{\OO}J(t,y)
 \|\widetilde g(t,y,v(y))\|_{\ell^2}^2\,dy .
\label{eq:H3-expanded}
\end{align}

The first term is controlled by the uniform ellipticity of $A$. 
For the second term, similar to the proof of \eqref{eq:weighted-drift-monotonicity}, we have
\begin{align*}
&-2\int_{\OO}J(t,y)\partial_t\rho_i(t,r(t,y))
 \partial_{y_i}v(y)v(y)\,dy
=-\int_{\OO}\partial_tJ(t,y)v(y)^2\,dy.
\end{align*}
Since $\partial_tJ/J\in L^\infty([0,T]\times\OO)$, it follows that
\begin{equation*}
-\int_{\OO}\partial_tJ(t,y)v(y)^2\,dy
\leq
\left\|\frac{\partial_tJ(t,\cdot)}{J(t,\cdot)}
\right\|_{L^\infty(\OO)}
|v|_t^2.
\end{equation*}

Finally, by the growth assumptions on $\widetilde f$ and $\widetilde g$,
Young's inequality, and the equivalence of $|\cdot|_t$ and $|\cdot|_0$, for
every $\varepsilon>0$,
\begin{align*}
&2\int_{\OO}J(t,y)\widetilde f(t,y,v(y))v(y)\,dy
+\int_{\OO}J(t,y)
 \|\widetilde g(t,y,v(y))\|_{\ell^2}^2\,dy\\
&\qquad\leq
\varepsilon\|v\|_V^2
+C_\varepsilon|v|_t^2
+C_\varepsilon\bigl(1+K_{0}^2\bigr).
\end{align*}
Choosing $\varepsilon<2\lambda_AJ_*$ and combining the preceding estimates,
we obtain
\begin{align}
&2{}_{V^*}\langle\iota_t^*\mathcal A(t,v),v\rangle_V
+\|\mathcal B(t,v)\|_{L_2(\ell^2,H_t)}^2
\leq
-c\|v\|_V^2
+C|v|_t^2
+C\bigl(1+K_{0}^2\bigr),
\label{eq:H3-corrected}
\end{align}
for some $c>0$ independent of $t$ and $v$. 


\emph{Verification of H4.}
By condition C3, we have for every $\psi\in V$,
\[
\|J(t,\cdot)\psi\|_V
\leq C\|\psi\|_V.
\]
Moreover, the boundedness of $A$, $q$, $J$, and $\nabla J$, together with
the growth assumption \eqref{eq:moving-growth}, implies
\begin{equation*}
\|\mathcal A(t,v)\|_{V^*}
\leq
C\left(
\|v\|_V
+\|\widetilde f(t,\cdot,v)\|_{L^2(\OO)}
\right)
\leq
C\left(
1+K_{0}
+\|v\|_V
\right).
\end{equation*}
Therefore,
\begin{align*}
\|\iota_t^*\mathcal A(t,v)\|_{V^*}
&=
\sup_{\|\psi\|_V\leq1}
\left|
{}_{V^*}\langle\mathcal A(t,v),J(t,\cdot)\psi\rangle_V
\right|\\
&\leq
C\left(
1+K_{0}
+\|v\|_V
\right).
\end{align*}
from which we obtain
\eqref{eq:H4-abstract} with $\beta=0$.

\emph{Verification of H5.}
Finally, \eqref{eq:moving-growth} and the uniform upper bound for $J$ imply
\begin{equation*}
\|\mathcal B(t,v)\|_{L_2(\ell^2,H_t)}^2
=
\int_{\OO}J(t,y)
\|\widetilde g(t,y,v(y))\|_{\ell^2}^2\,dy
\leq
C\bigl(K_{0}^2+|v|_t^2\bigr),
\end{equation*}
which is bounded by $h(t)(1+|v|_t^2)$ after one final enlargement of $C_0$.
Thus H5 holds.\vspace{10pt}

\noindent{\bf Completion of the proof of Theorem~\ref{thm:moving-wellposed}}\\
After the verification of $C1-C3$, $H1-H5$, we can now apply Theorem 2.8 in \cite{PanWangZhaiZhang2026} to obtain the existence and
uniqueness of the solution of equation \eqref{eq:abstract-fixed-equation}. The equivalence established
in Section~\ref{sec:solution-equivalence} transfers this result to the case of moving
domains. This completes the proof of Theorem~\ref{thm:moving-wellposed}.

\section{Dirichlet H\"older regularity of solutions of SPDEs on fixed domain}
\label{sec:fixed-regularity}
In this section, we establish a regularity result for solutions of SPDEs, which extends the results of \cite{HsuWangWang2017} to bounded domains with homogeneous Dirichlet boundary conditions and additional drift terms. The spatial domain is denoted by $D\subset\R^n$.

\subsection{Setting and statement of the result}
Let $D\subset \R^n$ be a bounded $C^{1,1}$ domain. We consider the stochastic parabolic equation
\begin{equation}
\left\{
\begin{aligned}
&du
=
\operatorname{div}(A\nabla u)\,dt
+
b(t,x,u)\cdot\nabla u\,dt
+
f(t,x,u)\,dt
+
g_i(t,x,u)\,dw_t^i,
&& (t,x)\in(0,\infty)\times D,
\\
&u
=0,
&& (t,x)\in(0,\infty)\times\partial D,
\\
&u(0)
=u_0,
&& x\in D.
\end{aligned}
\right.
\label{eq:dirichlet-spde}
\end{equation}

Here $\{w^i\}_{i\in\N}$ is a sequence of independent standard Brownian motions on a filtered probability space
$(\Omega,\mathcal F,\{\mathcal F_t\}_{t\geq0},\PP)$. 
The coefficient $g=(g_i)_{i\in\mathbb N}$ is $\ell^2$-valued, whereas
$b=b(t,x,r)$ is an $\mathbb R^n$-valued drift
coefficient. We assume that, for every fixed $x\in D$ and every
real-valued progressively measurable process
$h=(h_t)_{t\geq0}$, the processes
$
f(t,x,h_t),g(t,x,h_t),b(t,x,h_t)
$
are progressively measurable.

Consider the following assumptions:
\begin{assumption}\label{ass:fixed-domain-coefficients}
The coefficients satisfy the following conditions.

\begin{enumerate}[label=(\roman*)]
\item The matrix field $A=A(t,x,\omega)$ is progressively measurable and uniformly elliptic: there is $\lambda\in(0,1]$ such that
$$
\lambda |\xi|^2
\leq
 \xi^\top A(t,x,\omega) \xi
\leq
\lambda^{-1}|\xi|^2
$$
for $\xi\in\R^n$ and every $(t,x,\omega)\in(0,\infty)\times D\times \Omega$.

\item There are constants $\Lambda>0$ and a non-negative function
$K\in L^\infty(D)$ such that
$$
|f(t,x,r)|+\norm{g(t,x,r)}_{\ell^2}
\leq
K(x)+\Lambda |r|
$$
for every $(t,x,r)\in(0,\infty)\times D\times \R$.

\item There is a constant $B\geq0$ such that
$$
|b(t,x,r)|\leq B
$$
for every $(t,x,r)\in(0,\infty)\times D\times \R$.

\item The initial datum $u_0\in L^2(D)$ is deterministic.
\end{enumerate}
\end{assumption}

A process $u$ is called a variational solution of
\eqref{eq:dirichlet-spde} if
\begin{align*}
u &\in L^2\bigl(\Omega;C([0,T_0];L^2(D))\bigr)
\cap
L^2\bigl(\Omega\times(0,T_0);H_0^1(D)\bigr)
\end{align*}
for every $T_0>0$, and, for every deterministic
$\varphi\in H^1\bigl([0,T_0];H_0^1(D)\bigr)$,
\begin{align*}
\ip{u(t)}{\varphi(t)}
&=
\ip{u_0}{\varphi(0)}
-
\int_0^t
\ip{A(s)\nabla u(s)}{\nabla\varphi(s)}\,ds
+
\int_0^t
\ip{b(s,\cdot,u(s))\cdot\nabla u(s)}{\varphi(s)}\,ds\nonumber\\
&\quad
+
\int_0^t
\ip{f(s,\cdot,u(s))}{\varphi(s)}\,ds
+
\int_0^t
\ip{g_i(s,\cdot,u(s))}{\varphi(s)}\,dw_s^i
 +
\int_0^t\ip{u(s)}{\dot\varphi(s)}\,ds
\end{align*}
almost surely, for all $t\in[0,T_0]$. Here $\dot\varphi$ denotes the weak
time derivative.

We now state the main result of this section.
\begin{theorem}
\label{thm:main}
Let $u$ be the
variational solution of \eqref{eq:dirichlet-spde}
in the sense specified above. Under Assumption \ref{ass:fixed-domain-coefficients}, there exists
$$
\alpha=\alpha(n,\lambda,\Lambda,B,D)\in(0,1)
$$
such that for every $T_0>0$,
$$
u\in C^{\alpha/2,\alpha}([T_0,2T_0]\times\overline D)
\quad\text{almost surely.}
$$
Moreover, for every $p>0$, there exists a constant
$$
C=C(n,\lambda,\Lambda,B,T_0,p,D)
$$
such that
\begin{align}
\E \norm{u}_{C^{\alpha/2,\alpha}([T_0,2T_0]\times\overline D)}^p
&\leq
C
\left(
\norm{u_0}_{L^2(D)}
+
\norm{K}_{L^\infty(D)}
\right)^p .
\label{eq:holder-moment}
\end{align}
\end{theorem}

\vskip 0.4cm

The proof has two parts. First, the stochastic De Giorgi iteration gives
an $L^\infty$-moment bound up to the boundary. Second, one passes from
boundedness to H\"older regularity by decomposing the variational solution
into a solution of the stochastic heat equation and a solution of a parabolic equation with random coefficients.

\subsection{Stochastic De Giorgi iteration}

Fix $T_0>0$. For $k\geq0$, set
$$
I_k=\bigl[(1-2^{-k})T_0,2T_0\bigr].
$$
For $a\geq1$, define
$$
u_{k,a}
=
\bigl(u-a(1-2^{-k})\bigr)_+,
$$
and
$$
U_{k,a}^{D}
=
\left(
\int_{I_k}
\norm{u_{k,a}(t)}_{L^2(D)}^4\,dt
\right)^{1/2}.
$$
Notice that $U_{0,a}^{D}=\norm{u^+}_{L^4(0,2T_0;L^2(D))}^2$.


For $j\geq0$, define the martingale increment
\begin{align*}
X_{j,a}^{*,D}
=
\sup_{(1-2^{-j-1})T_0\leq s\leq t\leq 2T_0}
\int_s^t
\ip{g_i(\tau,\cdot,u(\tau))}{u_{j+1,a}(\tau)}_{L^2(D)}
\,dw_\tau^i .
\end{align*}
Thus $X_{k-1,a}^{*,D}$ is precisely the martingale term that appears in the energy estimate for $u_{k,a}$.

\begin{proposition}
\label{prop:degiorgi-step}
Assume
$
\norm{K}_{L^\infty(D)}\leq1.
$
If $n\geq3$, then there is a constant $C=C(n,\lambda,\Lambda,B,T_0,D)$ such that
\begin{align}
U_{k,a}^{D}
&\leq
\frac{C^k}{a^{2/(n+1)}}
\left(
U_{k-1,a}^{D}
+
X_{k-1,a}^{*,D}
\right)
\left(U_{k-1,a}^{D}\right)^{1/(n+1)} .
\label{eq:degiorgi-iteration}
\end{align}
For $n=1,2$ and every $\mu\in(0,1/3)$, the same estimate holds with $1/(n+1)$ replaced by $\mu$ and $2/(n+1)$ replaced by $2\mu$.
\end{proposition}

\begin{proof}
By H\"older's inequality,
\begin{align}
\norm{u_{k,a}(t)}_{L^2(D)}^2
&\leq
\norm{u_{k,a}(t)}_{L^{2(n+1)/n}(D)}^2
|\{x\in D:u_{k,a}(t,x)>0\}|^{1/(n+1)} .
\label{eq:holder-level}
\end{align}
Since
$$
\{u_{k,a}>0\}
\subset
\{u_{k-1,a}>2^{-k}a\},
$$
Chebyshev's inequality gives
$$
|\{u_{k,a}(t)>0\}|
\leq
\left(\frac{2^k}{a}\right)^2
\norm{u_{k-1,a}(t)}_{L^2(D)}^2 .
$$
Squaring \eqref{eq:holder-level}, integrating over $I_k$, and applying H\"older's inequality in time yield
\begin{align}
U_{k,a}^{D}
&\leq
\left(\frac{2^k}{a}\right)^{2/(n+1)}
\norm{u_{k,a}}_{L^{4(n+1)/n}(I_k;L^{2(n+1)/n}(D))}^2
\left(U_{k-1,a}^{D}\right)^{1/(n+1)} .
\label{eq:pre-interpolation}
\end{align}
Assume first that $n\geq3$, and set
\[
q:=\frac{2n}{n-2},
\qquad
r:=\frac{4(n+1)}{n},
\qquad
s:=\frac{2(n+1)}{n},
\qquad
\theta:=\frac{n}{2(n+1)}.
\]
Since $u_{k,a}(t,\cdot)\in H_0^1(D)$ for almost every
$t\in I_k$, the Sobolev inequality implies
\begin{equation*}
\|u_{k,a}\|_{L^2(I_k;L^q(D))}
\leq
C\|\nabla u_{k,a}\|_{L^2(I_k\times D)}.
\end{equation*}
Indeed, for almost every $t$, one may extend
$u_{k,a}(t,\cdot)$ by setting $u_{k,a}(t,\cdot)=0$ outside $D$, and apply the Sobolev inequality on
$\mathbb R^n$.

Moreover, the choice of $\theta$ gives
\[
\frac{1}{s}
=
\frac{1-\theta}{2}+\frac{\theta}{q},
\qquad
\frac{1}{r}=\frac{\theta}{2}.
\]
Therefore, interpolation between
$L^\infty(I_k;L^2(D))$ and $L^2(I_k;L^q(D))$ yields
\begin{align}
\|u_{k,a}\|_{L^r(I_k;L^s(D))}
&\leq
C
\|u_{k,a}\|_{L^\infty(I_k;L^2(D))}^{1-\theta}
\|u_{k,a}\|_{L^2(I_k;L^q(D))}^{\theta}
\nonumber\\
&\leq
C
\|u_{k,a}\|_{L^\infty(I_k;L^2(D))}^{1-\theta}
\|\nabla u_{k,a}\|_{L^2(I_k\times D)}^{\theta}.
\label{eq:mixed-norm-interpolation}
\end{align}
Squaring \eqref{eq:mixed-norm-interpolation} and applying the weighted
Young inequality, we obtain
\begin{align}
\|u_{k,a}\|_{L^{\frac{4(n+1)}{n}}
 (I_k;L^{\frac{2(n+1)}{n}}(D))}^2
&\leq
C\left[
\sup_{t\in I_k}
\|u_{k,a}(t)\|_{L^2(D)}^2
+
\int_{I_k}
\|\nabla u_{k,a}(t)\|_{L^2(D)}^2\,dt
\right].
\label{eq:parabolic-interpolation}
\end{align}

We now estimate the two terms on the right-hand side of
\eqref{eq:parabolic-interpolation}. Set
\[
\Psi_{k,a}(r)=|(r-a(1-2^{-k}))_+|^2,
\]
The function $\Psi_{k,a}$ is of class $C^1$, while its second derivative has a jump discontinuity at $r=a(1-2^{-k})$. The application of It\^o's formula below is therefore understood through a standard regularization argument: we approximate $\Psi_{k,a}$ by smooth convex functions $\Psi_{k,a}^{\varepsilon}$ and regularize the weak equation in the spatial variable by an approximation of the identity. Applying It\^o's formula to the regularized equation and then letting first the spatial regularization parameter and subsequently $\varepsilon$ tend to zero, as in \cite[Remark~2.3]{HsuWangWang2017}, gives
\begin{align}
d\norm{u_{k,a}(t)}_{L^2(D)}^2
&=
-2\ip{A\nabla u_{k,a}}{\nabla u_{k,a}}\,dt
+
2\ip{g_i(u)}{u_{k,a}}\,dw_t^i
\nonumber\\
&\quad
+
2\int_D
u_{k,a}\,
b(t,x,u)\cdot\nabla u_{k,a}\,dx\,dt
\nonumber\\
&\quad
+
\int_D
\left(
\norm{g(u)}_{\ell^2}^2
+
2u_{k,a}f(u)
\right)
\1_{\{u_{k,a}>0\}}\,dx\,dt .
\label{eq:ito-truncation}
\end{align}
By the uniform ellipticity assumption of $A$, we have 
$$
\ip{A\nabla u_{k,a}}{\nabla u_{k,a}}
\geq
\lambda\norm{\nabla u_{k,a}}_{L^2(D)}^2 .
$$
On the set $\{u_{k,a}>0\}$, one has
$$
1\leq a\leq 2^k u_{k-1,a},
\qquad
0<u\leq u_{k-1,a}+a\leq (1+2^k)u_{k-1,a}.
$$
By Young's inequality and the boundedness of $b$, we have
$$
2\left|\int_D u_{k,a}\,b(t,x,u)\cdot\nabla u_{k,a}\,dx\right|
\leq
\frac{\lambda}{2}\norm{\nabla u_{k,a}}_{L^2(D)}^2
+
C B^2\norm{u_{k,a}}_{L^2(D)}^2 .
$$
Using the linear growth condition and the fact that $u_{k,a}\leq u_{k-1,a}$, and
$\norm{K}_{L^\infty(D)}\leq1$, the last term on the right of \eqref{eq:ito-truncation} is bounded by
$$
C^k\norm{u_{k-1,a}(t)}_{L^2(D)}^2\,dt .
$$
Integrating \eqref{eq:ito-truncation} from $t_0\in I_{k-1}\setminus I_k$ to $t\in I_k$, and then taking the supremum over $t\in I_k$, yields
\begin{align}
\sup_{t\in I_k}\norm{u_{k,a}(t)}_{L^2(D)}^2
&+
\int_{I_k}\norm{\nabla u_{k,a}(t)}_{L^2(D)}^2\,dt\nonumber\\
&\leq
C\norm{u_{k,a}(t_0)}_{L^2(D)}^2
+
C^k U_{k-1,a}^{D}
+
C X_{k-1,a}^{*,D}.
\label{eq:energy-after-ito}
\end{align}
By averaging over $I_{k-1}\setminus I_k$, and using $u_{k,a}\leq u_{k-1,a}$, we may choose $t_0$ such that
$$
\norm{u_{k,a}(t_0)}_{L^2(D)}^2
\leq
C^k U_{k-1,a}^{D}.
$$
Combining this estimate with \eqref{eq:pre-interpolation}, \eqref{eq:parabolic-interpolation}, and \eqref{eq:energy-after-ito} yields \eqref{eq:degiorgi-iteration}.
\end{proof}
\begin{remark}
When $n=1$ or $n=2$, the critical Sobolev embedding used above is not
available in the form
\[
H_0^1(D)\hookrightarrow L^{\frac{2n}{n-2}}(D).
\]
Instead, one uses the embedding
\[
H_0^1(D)\hookrightarrow L^q(D)
\]
for any finite $q$ (and, in dimension one, also
$H_0^1(D)\hookrightarrow L^\infty(D)$). Consequently, for every
$\mu\in(0,1/3)$, there exists a constant
$C=C(n,\lambda,\Lambda,B,T_0,D,\mu)$ such that
\begin{equation}
U_{k,a}^{D}
\leq
\frac{C^k}{a^{2\mu}}
\left(
U_{k-1,a}^{D}+X_{k-1,a}^{*,D}
\right)
\left(U_{k-1,a}^{D}\right)^\mu .
\label{eq:degiorgi-iteration-low-dimensional}
\end{equation}
The subsequent tail argument only requires that the exponent of
$U_{k-1,a}^{D}$ be strictly positive. Hence the proof proceeds in the same
way as for $n=1$ and $n=2$.
\end{remark}

Proposition~\ref{prop:degiorgi-step} provides the recursive estimate required for the iteration. Straightforwardly adapting the argument of \cite[Section~3]{HsuWangWang2017} to the present setting yields the following result.

\begin{proposition}
\label{prop:hww-tail}
Assume $T_0=1$ and $\norm{K}_{L^\infty(D)}\leq1$. There are constants
$M_0=M_0(n,\lambda,\Lambda,B,D)>0$ and $\delta>0$ such that for all $a\geq1$ and $M>M_0$,
\begin{align}
\PP\left(
\norm{u^+}_{L^\infty([1,2]\times D)}>a,\,
M\norm{u^+}_{L^4(0,2;L^2(D))}\leq a
\right)
&\leq
\exp(-M^\delta).
\label{eq:hww-tail-dirichlet}
\end{align}
For $n\geq3$, one may take $\delta=1/(n+1)$. For $n=1,2$, $\delta$ may be any number in $(0,1/3)$.
\end{proposition}

\begin{proposition}
\label{prop:linfty}
For every $p>0$ and $T_0>0$, there exists
$C=C(n,\lambda,\Lambda,B,T_0,p,D)$ such that
\begin{align*}
\E\norm{u}_{L^\infty([T_0,2T_0]\times D)}^p
&\leq
C
\left(
\norm{u_0}_{L^2(D)}
+
\norm{K}_{L^\infty(D)}
\right)^p .
\end{align*}
\end{proposition}

\begin{proof}
By scaling it suffices to consider the case
$
T_0=1,\,
\norm{u_0}_{L^2(D)}+\norm{K}_{L^\infty(D)}\leq1.
$
Set
$$
\Phi(t)=\norm{u(t)}_{L^2(D)}^2+1.
$$
By It\^o's formula, we have
\begin{align}\label{ito}
d\Phi(t)
&=
\Phi(t)\bigl(F(t)\,dt+dG_t\bigr),
\end{align}
where
\begin{align*}
F(t)
&=
\frac{
-2\ip{A\nabla u}{\nabla u}
+
2\ip{b(u)\cdot\nabla u}{u}
+
2\ip{f(u)}{u}
+
\norm{g(u)}_{L^2(D;\ell^2)}^2
}{
\norm{u}_{L^2(D)}^2+1
},
\\
G_t
&=
\int_0^t
\frac{
2\ip{g_i(u(s))}{u(s)}
}{
\norm{u(s)}_{L^2(D)}^2+1
}
\,dw_s^i .
\end{align*}
The solution of \eqref{ito} is given by
\begin{equation*}
\Phi(t)=\Phi(0)
\exp\left(
\int_0^t F(s)\,ds+G_t-\frac12\langle G\rangle_t
\right).
\end{equation*}
By the assumptions and the fact that
$$
|\ip{b(u)\cdot\nabla u}{u}|
\leq
\frac{\lambda}{2}\norm{\nabla u}_{L^2(D)}^2
+
C B^2\norm{u}_{L^2(D)}^2
$$
we obtain $F(t)\leq C$ and
$\langle G\rangle_t\leq C t$ on $[0,2]$. 
Since $\langle G\rangle_t\leq Ct$ on $[0,2]$, Novikov's condition holds, which implies
the exponential process
$$
\exp\left(
rG_{t}-\frac{r^2}{2}\langle G\rangle_{t}
\right)
$$
is a martingale. Note that $F\leq C$ and
$\langle G\rangle_t\leq Ct$, we have
\begin{align*}
\E\Phi(t)^r
&\leq
\Phi(0)^r 
\E\left[
\exp\left(
r\int_0^tF(s)\,ds+rG_t-\frac r2\langle G\rangle_t
\right)
\right]
\leq
C_r\Phi(0)^r ,
\end{align*}
uniformly for $t\in[0,2]$. Taking $r=q/2$ and using
$\|u(t)\|_{L^2(D)}^q\leq \Phi(t)^{q/2}$ gives
\begin{align}
\E\int_0^2\norm{u(t)}_{L^2(D)}^q\,dt
&\leq
C_q .
\label{eq:l2-moment}
\end{align}
Set
$$
X=\norm{u}_{L^\infty([1,2]\times D)},
\qquad
Y=\norm{u}_{L^4(0,2;L^2(D))}.
$$
Applying Proposition \ref{prop:hww-tail} to $u$ and to $-u$ gives
\begin{align*}
\PP(X>a,\;Y\leq a/M)
&\leq
2e^{-M^\delta}
\end{align*}
for all $a\geq1$ and $M\geq M_0$. Taking $M=\sqrt a$ for large $a$, and using \eqref{eq:l2-moment} to control moments of $Y$, one obtains
\begin{align*}
\E X^p
&=
p\int_0^\infty a^{p-1}\PP(X>a)\,da
\nonumber\\
&\leq
C
+
p\int_{M_0^2}^{\infty}
a^{p-1}\PP(Y>\sqrt a)\,da
+
2p\int_{M_0^2}^{\infty}
a^{p-1}\PP(X>a, Y\le \sqrt{a})\,da
\nonumber\\
&\leq
C
+
\E Y^{2p}
+
2p\int_{M_0^2}^{\infty}
a^{p-1}e^{-a^{\delta/2}}\,da
\leq
C_p .
\end{align*}
Thus, we finish the proof.
\end{proof}

The last step uses two standard boundary estimates. The first estimate is quoted from \cite[Theorem 2.7]{Kim2004}.
\begin{lemma}
\label{lem:dirichlet-heat}
Let $D\subset\R^n$ be bounded and $C^{1,1}$, and let $p>n+2$. Let $v$ solve

\begin{equation*}
\left\{
\begin{aligned}
&dv
=
\Delta_D v\,dt+h_i\,dw_t^i,
&& (t,x)\in(S,2T_0]\times D,
\\
&v
=0,
&& (t,x)\in(S,2T_0]\times\partial D,
\\
&v(S)
=0,
&& x\in D.
\end{aligned}
\right.
\end{equation*}
If $h$ is predictable and
$$
h\in L^p\bigl(\Omega\times(S,2T_0]\times D;\ell^2\bigr),
$$
then there are $\alpha_1\in(0,1)$ and $C>0$, depending only on $n,p,T_0,D$, such that
\begin{align*}
\E\|v\|_{C^{\alpha_1/2,\alpha_1}([T_0,2T_0]\times\overline D)}^p
+
\E\|\nabla v\|_{L^p((T_0,2T_0]\times D)}^p
&\leq
C\E\|h\|_{L^p((S,2T_0]\times D;\ell^2)}^p .
\end{align*}
Moreover,
\begin{align*}
\E\|D^2v\|_{L^p((T_0,2T_0];W^{-1,p}(D))}^p
&\leq
C\E\|h\|_{L^p((S,2T_0]\times D;\ell^2)}^p .
\end{align*}
\end{lemma}

The second estimate goes back to \cite[Chapter III, Theorem 10.1]{LadyzhenskayaSolonnikovUraltseva1968}. For the case $c=0$, we also refer to \cite[Theorem 3.2]{DebusscheDeMoorHofmanova2015}, where a detailed proof is provided.

\begin{lemma}
\label{lem:boundary-dgn}
Let $D\subset\R^n$ be bounded and $C^{1,1}$. Suppose $z$ solves
\begin{align*}
\partial_t z
&=
\operatorname{div}(A\nabla z)+c\cdot\nabla z+F+\operatorname{div} H
&&\text{in }(S,2T_0]\times D,
\\
z&=0
&&\text{on }(S,2T_0]\times\partial D,
\end{align*}
where $A$ is bounded and uniformly elliptic, and $c\in L^\infty((S,2T_0]\times D;\R^n)$. Assume
$$
z\in L^\infty((S,2T_0]\times D),
\qquad
F\in L^p((S,2T_0]\times D),
\qquad
H\in L^p((S,2T_0]\times D;\R^n)
$$
for some $p>n+2$. Then there are $\alpha_2\in(0,1)$ and $C>0$, depending only on $n,\lambda,p,T_0,D$ and $\|c\|_{L^\infty}$, such that
\begin{align*}
\|z\|_{C^{\alpha_2/2,\alpha_2}([T_0,2T_0]\times\overline D)}
&\leq
C
\left(
\|z\|_{L^\infty((S,2T_0]\times D)}
+
\|F\|_{L^p((S,2T_0]\times D)}
+
\|H\|_{L^p((S,2T_0]\times D)}
\right).
\end{align*}
\end{lemma}

\subsection{Proof of Theorem \ref{thm:main}}

\begin{proof}
Fix $T_0>0$ and set $S=T_0/2$. Let $v$ be the variational solution of the
Dirichlet stochastic heat equation displayed in
Lemma~\ref{lem:dirichlet-heat}:
\begin{equation*}
\left\{
\begin{aligned}
&dv
=
\Delta_D v\,dt
+
g_i(t,x,u(t,x))\,dw_t^i,
&& (t,x)\in(S,2T_0]\times D,
\\
&v
=0,
&& (t,x)\in(S,2T_0]\times\partial D,
\\
&v(S)
=0,
&& x\in D.
\end{aligned}
\right.
\end{equation*}
By the linear growth condition,
$$
\|g(t,\cdot,u(t))\|_{L^p(D;\ell^2)}
\leq
\|K\|_{L^p(D)}+\Lambda\|u(t)\|_{L^p(D)}.
$$
Since $D$ is bounded and Proposition~\ref{prop:linfty} gives $u\in L^p(\Omega;L^\infty([S,2T_0]\times D))$ for every finite $p$,
\begin{align*}
\E\|g(\cdot,\cdot,u)\|_{L^p((S,2T_0]\times D;\ell^2)}^p
&\leq
C
\left(
\norm{u_0}_{L^2(D)}
+
\norm{K}_{L^\infty(D)}
\right)^p .
\end{align*}
By Lemma~\ref{lem:dirichlet-heat}, we have
\begin{align}
&\E\|v\|_{C^{\alpha_1/2,\alpha_1}([T_0,2T_0]\times\overline D)}^p
+
\E\|\nabla v\|_{L^p((T_0,2T_0]\times D)}^p
+
\E\|D^2v\|_{L^p((T_0,2T_0];W^{-1,p}(D))}^p
\nonumber\\
&\qquad\leq
C
\left(
\norm{u_0}_{L^2(D)}
+
\norm{K}_{L^\infty(D)}
\right)^p .
\label{eq:v-estimate}
\end{align}
Set $\phi=u-v$. Then we have
\begin{align}
\partial_t\phi
&=
\operatorname{div}(A\nabla\phi)
+
b(t,x,u)\cdot\nabla\phi
+
f(t,x,u)
+
b(t,x,u)\cdot\nabla v
+
\operatorname{div}((A-I)\nabla v),
\label{eq:phi-equation}
\\
\phi&=0
\qquad\text{on }(S,2T_0]\times\partial D.
\nonumber
\end{align}
For each fixed
$\omega$ outside a null set, \eqref{eq:phi-equation} is a deterministic
divergence-form equation. In Lemma~\ref{lem:boundary-dgn} we take
$$
z=\phi,
\qquad
c=b(\cdot,\cdot,u),
\qquad
F=f(\cdot,\cdot,u)+b(\cdot,\cdot,u)\cdot\nabla v,
\qquad
H=(A-I)\nabla v.
$$
The coefficient $c$ is bounded by assumption. The function $z$ is bounded because $u$ is bounded by Proposition~\ref{prop:linfty} and $v$ is H\"older continuous by \eqref{eq:v-estimate}. Moreover,
$$
|f(t,x,u)|\leq K(x)+\Lambda |u(t,x)|,
\qquad
|b(t,x,u)\cdot\nabla v|\leq B|\nabla v|,
$$
and therefore
\begin{align}
\E\|F\|_{L^p((S,2T_0]\times D)}^p
&\leq
C
\left(
\norm{u_0}_{L^2(D)}
+
\norm{K}_{L^\infty(D)}
\right)^p.
\label{eq:f-estimate}
\end{align}
Since $A$ is bounded,
\begin{equation}
\|H\|_{L^p((S,2T_0]\times D)}
\leq C\|\nabla v\|_{L^p((S,2T_0]\times D)}.
\label{eq:H-estimate}
\end{equation}
Using Lemma~\ref{lem:boundary-dgn}, \eqref{eq:v-estimate}, \eqref{eq:f-estimate}, \eqref{eq:H-estimate}, and Proposition~\ref{prop:linfty}, we get
\begin{align}
\E\|\phi\|_{C^{\alpha_2/2,\alpha_2}([T_0,2T_0]\times\overline D)}^p
&\leq
C
\left(
\norm{u_0}_{L^2(D)}
+
\norm{K}_{L^\infty(D)}
\right)^p .
\label{eq:phi-holder}
\end{align}
Recall that $u=\phi+v$. With $\alpha=\min\{\alpha_1,\alpha_2\}$, \eqref{eq:v-estimate} and \eqref{eq:phi-holder} imply \eqref{eq:holder-moment}.
\end{proof}

\section{Boundary regularity on moving domains}
\label{sec:moving-regularity}

\begin{proof}[Proof of Theorem~\ref{thm:physical-holder-transfer}]
Existence and uniqueness follow from Theorem~\ref{thm:moving-wellposed}. It
remains to prove the H\"older estimate. Define
\[
\widehat{u}(t,y):=u(t,r(t,y)),
\qquad y\in\OO.
\]
By Section~\ref{sec:solution-equivalence}, $\widehat{u}$ is the variational solution
of \eqref{eq:abstract-fixed-equation}, which is precisely equation \eqref{eq:fixed-spde} on $\OO$. We now verify that its coefficients satisfy
the assumptions of Theorem~\ref{thm:main}.

First we check that $A(t,y)=K_r(t,y)K_r(t,y)^\top$ is uniformly elliptic. Indeed,
$$
a^{ij}(t,y)\xi_i\xi_j=|K_r(t,y)^\top\xi|^2.
$$
Assumption~\ref{ass:moving-domain} and \eqref{eq:K-inverse} imply that both $D_y r$ and $(D_y r)^{-1}$ are uniformly bounded. Hence there are constants $0<c<C<\infty$ such that
$$
c|\xi|^2\leq a^{ij}(t,y)\xi_i\xi_j\leq C|\xi|^2
$$
for all $\xi\in\R^d$, uniformly in $(t,y)$.

The first-order coefficient
$$
q^i
=
\Delta_x\rho_i(t,r(t,y))
-\partial_t\rho_i(t,r(t,y))
-\partial_{y_j}a^{ij}(t,y)
$$
is bounded because of the Assumption~\ref{ass:moving-domain} and the definition of $a^{ij}$. Thus $q\in L^\infty((0,T)\times\OO;\R^d)$.

The growth conditions of $\widetilde f$ and $\widetilde g$ are preserved by the pullback. By \eqref{eq:fg-pullback} and \eqref{eq:moving-growth},
$$
|\widetilde f(t,y,z)|+\|\widetilde g(t,y,z)\|_{\ell^2}
=
|f(t,r(t,y),z)|+\|g(t,r(t,y),z)\|_{\ell^2}
\leq
K_{0}+\Lambda |z|.
$$
The boundary condition is homogeneous on the boundary $\partial\OO$,
and $u_0$ is deterministic in $L^2(\OO)$. Hence Theorem~\ref{thm:main},
applied with $D=\OO$, $n=d$, and $b=q$, gives an exponent
$\alpha\in(0,1)$ and, for every $p>0$,
\begin{equation}
\E\|\widehat{u}\|_{C^{\alpha/2,\alpha}([T_0,2T_0]\times\overline\OO)}^p
\leq
C
\left(
\|u_0\|_{L^2(\OO)}
+K_{0}
\right)^p.
\label{eq:moving-holder-estimate}
\end{equation}

For almost every $\omega$, the path $\widehat{u}(\omega)$ is H\"older continuous on $[T_0,2T_0]\times\overline\OO$. For $(t,x)\in Q_{T_0,2T_0}$, the space-variable representative of the solution is
$$
u(t,x,\omega)=\widehat{u}(t,\rho(t,x),\omega).
$$
Fix two points $(t,x),(s,z)\in Q_{T_0,2T_0}$ and set
$$
y=\rho(t,x),\qquad \zeta=\rho(s,z).
$$
Then $y,\zeta\in\overline\OO$, and
$$
u(t,x)-u(s,z)=\widehat{u}(t,y)-\widehat{u}(s,\zeta).
$$
The H\"older continuity of $\widehat{u}$ gives
\begin{equation}
|u(t,x)-u(s,z)|
\leq
\|\widehat{u}\|_{C^{\alpha/2,\alpha}([T_0,2T_0]\times\overline\OO)}
\left(|t-s|^{\alpha/2}+|y-\zeta|^\alpha\right).
\label{eq:transfer-first}
\end{equation}
It remains to compare the distance $ |y-\zeta| $ with the space distance $ |x-z| $.

Since $x=r(t,y)$ and $z=r(s,\zeta)$, we have
\begin{align}
 |y-\zeta|
&=
|\rho(t,x)-\rho(t,r(t,\zeta))|
\nonumber\\
& \leq
\|D_x\rho(t,\cdot)\|_{L^\infty(\mathcal O_t)}
|x-r(t,\zeta)|
\nonumber\\
&\leq
C\left(|x-z|+|z-r(t,\zeta)|\right)
\nonumber\\
&=
C\left(|x-z|+|r(s,\zeta)-r(t,\zeta)|\right)
\nonumber\\
&\leq
C\bigl(|x-z|+|t-s|\bigr).
\label{eq:reference-distance-bound}
\end{align}
In the last step we used the uniform bound on $\partial_t r$ and the fact that
$$|r(s,\zeta)-r(t,\zeta)|\leq \int_s^t\|\partial_\tau r(\tau,\cdot)\|_{L^\infty(\OO)}\,d\tau\leq C|t-s|.$$
Thus, by \eqref{eq:reference-distance-bound} we have,
$$
|y-\zeta|^\alpha
\leq
C\bigl(|x-z|^\alpha+|t-s|^\alpha\bigr)
\leq
C\bigl(|x-z|^\alpha+|t-s|^{\alpha/2}\bigr),
$$
where the last inequality uses $t,s\in[0,T]$. Substituting this into \eqref{eq:transfer-first} gives
\begin{equation*}
|u(t,x)-u(s,z)|
\leq
C\|\widehat{u}\|_{C^{\alpha/2,\alpha}([T_0,2T_0]\times\overline\OO)}
\left(|t-s|^{\alpha/2}+|x-z|^\alpha\right).
\end{equation*}
Since $\rho(t,\cdot)$ maps $\overline{\mathcal O_t}$ onto
$\overline{\OO}$ for every $t$, we have
$$
\sup_{(t,x)\in Q_{T_0,2T_0}}|u(t,x)|
=
\sup_{(t,y)\in[T_0,2T_0]\times\overline\OO}|\widehat{u}(t,y)|
\leq
\|\widehat{u}\|_{C^{\alpha/2,\alpha}([T_0,2T_0]\times\overline\OO)} .
$$
Therefore, we have
$$
\|u\|_{C^{\alpha/2,\alpha}(Q_{T_0,2T_0})}
\leq
C
\|\widehat{u}\|_{C^{\alpha/2,\alpha}([T_0,2T_0]\times\overline\OO)}
\quad\text{almost surely}.
$$
Taking $p$-th moments and using \eqref{eq:moving-holder-estimate}, we obtain
\eqref{eq:physical-holder-moment}.
\end{proof}

\subsection*{Acknowledgments}
The authors would like to thank Dr Tianyi Pan for insightful discussions and valuable suggestions. This work is partially supported by the National Key R\&D Program of China (No. 2022YFA1006001), the National Natural Science Foundation of China (Nos. 12131019, 12571158) and China Postdoctoral Science Foundation (No. 2026M793409).


\end{document}